\documentclass[hidelinks,12pt]{article}
\usepackage{amsfonts,amsmath,amssymb,amsthm}
\usepackage{mathtools}
\usepackage{graphicx} 
\usepackage[textsize=tiny]{todonotes}
\usepackage{hyperref,cleveref}
\usepackage{microtype}

\newcommand{\Z}{\mathbb{Z}}

\newcommand{\F}{\mathbb{F}}

\newtheorem{theorem}{Theorem}[section]
\newtheorem{lemma}[theorem]{Lemma}

\newtheorem{definition}[theorem]{Definition}
\newtheorem{proposition}[theorem]{Proposition}
\newtheorem{corollary}[theorem]{Corollary}

\newtheoremstyle{BoldRemark} 
{\topsep}                    
{\topsep}                    
{\upshape}                   
{}                           
{\bfseries}                  
{.}                          
{.5em}                       
{}  
\theoremstyle{BoldRemark}

\newtheorem{example}[theorem]{Example}

\date{}
\author{Noah Wisdom}
\title{A classification of $C_{p^n}$-Tambara fields}
\begin{document}
\maketitle
\begin{abstract}
	Tambara functors arise in equivariant homotopy theory as the structure adherent to the homotopy groups of a coherently commutative equivariant ring spectrum. We show that if $k$ is a field-like $C_{p^n}$-Tambara functor, then $k$ is the coinduction of a field-like $C_{p^s}$-Tambara functor $\ell$ such that $\ell(C_{p^s}/e)$ is a field. If this field has characteristic other than $p$, we observe that $\ell$ must be a fixed-point Tambara functor, and if the characteristic is $p$, we determine all possible forms of $\ell$ through an analysis of the behavior of the Frobenius endomorphism and the trace of a $C_p$-Galois extension.
\end{abstract}
\tableofcontents

\section{Introduction}

For $G$ a finite group, $G$-Tambara functors are the basic objects of study in equivariant algebra. They arise in homotopy theory as the natural structure adherent to the homotopy groups of a $G$-$\mathbb{E}_\infty$ ring spectrum, though they additionally arise through many important situations in commutative algebra. For example any finite Galois field extension gives rise to a $Gal$-Tambara functor, and the representation rings of $G$ and its subgroups naturally have the structure of a Tambara functor.

Roughly speaking, the notion of a $G$-Tambara functor is obtained by abstracting the notion of a Galois extension with Galois group $G$. More precisely, in this setting, one has intermediate fields for each subgroup $H \subset G$ which have residual Weyl group $W_G H$ action, contravariant inclusions between intermediate fields, as well as covariant transfer and norm maps between intermediate fields, all satisfying formulae relating various compositions. In a $G$-Tambara functor, we ask merely for rings $k(G/H)$ for each subgroup of $G$, and do not require that restriction maps are inclusions. Here we still have transfers, norms, and Weyl group actions, whose compositions satisfy similar formulae. A morphism of $G$-Tambara functors is a collection of ring maps, one for each level $G/H$, which commute with restrictions, norms, transfers, and Weyl group actions.

While $G$-Tambara functors are the equivariant algebra analogues to rings, Nakaoka \cite{Nak11a} \cite{Nak11b} has defined field-like Tambara functors as those nonzero $k$ for which every morphism $k \rightarrow \ell$ with $\ell \neq 0$ is monic. In particular, Nakaoka defines an ideal of a Tambara functor and shows that every Nakaoka ideal is obtained as the collection of kernels at each level of a map of $G$-Tambara functors. Next, Nakaoka observes \cite[Theorem 4.32]{Nak11a} that $k$ is field-like if and only if $k(G/e)$ has no nontrivial $G$-invariant ideals and all restriction maps in $k$ are injective. Additionally, upcoming work of Schuchardt, Spitz, and the author \cite{SSW24} classify the algebraically closed (or Nullstellensatzian) fields in Tambara functors: they are precisely the coinductions of algebraically closed fields.

Fields play an important role in homotopy theory and higher algebra; the rings $\F_p$ are among the most fundamental objects, viewed as $\mathbb{E}_\infty$-ring spectra via the Eilenberg-MacLane construction. While this construction makes sense for any discrete ring, the most powerful computational tools of this form are usually obtained by feeding in a field. In equivariant homotopy theory, there is a similar Eilenberg-MacLane construction, although in the literature, computations are typically carried out with respect to the constant Tambara functors associated to fields (or the initial Tambara functor). These are indeed field-like Tambara functors, although they do not have the property that all of their Mackey functor modules are free! On the other hand, there are many other Tambara fields, for which there are relatively few computations in the literature, which do have the property that all of their Mackey functor modules are free (namely those which are coinduced from fields). We hope that the results of this article will serve as a source of inspiration for equivariant computations. For example, we pose the following question: what are the $RO(C_{p^n})$-graded homotopy groups of all $C_{p^n}$-Tambara fields?

We aim to give a complete classification of the field-like $C_{p^n}$-Tambara functors, for $C_{p^n}$ the cyclic group of order $p^n$. The impetus of this work is the following observation of David Chan and Ben Spitz \cite{CS24}. They showed that if $k$ is field-like, then $k(G/e)$ is a product of copies of a field $\F$ permuted transitively by the $G$-action. Despite the fact that this may be deduced relatively quickly from Nakaoka's results, it suggests that an enormous amount of structure on a Tambara functor is forced by the field-like condition. To capture the special case of the Chan-Spitz result for which $k(G/e)$ is a field, we introduce the following definition.

\begin{definition}
Let $k$ be a field-like $G$-Tambara functor. If $k(G/e) \cong Fun(G/H,R)$ for some $H$-ring $R$ and proper subgroup $H \subset G$, we call $k$ separated. Otherwise we call $k$ clarified. 
\end{definition}

The word ``clarified" is meant to evoke the mental picture of clarified butter. The source of the terminology arises from future work of the author, in which a notion of ``clarified" $G$-Tambara functor is defined which generalizes the above definition. Additionally, a functor from $G$-Tambara functors to clarified $G$-Tambara functors is constructed, which the author calls the ``clarification" functor.

If a field-like Tambara functor $k$ is separated, we may express this suggestively as $k(G/e) \cong \textrm{Coind}_H^G \F$, where $\textrm{Coind}_H^G$ is the coinduction functor from $H$-rings to $G$-rings, right adjoint to the restriction morphism. A similar right adjoint exists on the level of Tambara functors, also called coinduction and written $\textrm{Coind}_H^G$. To reduce clutter, we introduce the notation $\textrm{Coind}_i^n$ for the coinduction from $C_{p^i}$ to $C_{p^n}$ (and $\textrm{Res}_i^n$ for the restriction from $C_{p^n}$ to $C_{p^i}$).

\begin{theorem}
	For $C_{p^n}$ the cyclic order $p^n$ group, if $k$ is a field-like $C_{p^n}$-Tambara functor, then \[ k \cong \textnormal{Coind}_i^n \ell \] for some clarified $C_{p^i}$-Tambara functor $\ell$.
\end{theorem}

This reduces the classification problem of $C_{p^n}$-Tambara fields to those which are clarified, ie have $C_{p^n}/e$ level a field. If the characteristic of this field is prime to $p$, or the $C_{p^n}$-fixed point subfield is perfect, then the classification of such Tambara fields is straightforward.

\begin{theorem}
\label{thm:two-cases-of-clarified-fields}
	Suppose $k$ is a clarified field-like $C_{p^n}$-Tambara functor. If either
\begin{enumerate}
	\item the characteristic of $k(C_p/e)$ is not $p$, or
	\item $k(C_{p^n}/e)^{C_{p^n}}$ is perfect
\end{enumerate} then $k$ is canonically isomorphic to the fixed-point Tambara functor associated to $k(C_{p^n}/e)$, ie the restriction maps determine isomorphisms $k(C_{p^n}/C_{p^i}) \rightarrow k(C_{p^n}/e)^{C_{p^i}}$.
\end{theorem}

In the case of characteristic $p$ with $k(C_{p^n}/e)^{C_{p^n}}$ nonperfect, it turns out we may still classify all possible structure. Roughly speaking, a clarified field-like $C_{p^n}$-Tambara functor is obtained by choosing a descending collection of subrings of a field with $C_{p^n}$-action. The chief obstruction to an arbitrary collection of subrings forming a $C_{p^n}$-Tambara functor is that they must contain the appropriate norms and transfers. In particular, one may first show that all such subrings must be subfields which contain the image the $C_{p^n}$-fixed points under the $n$th iterate of the Frobenius endomorphism.

With this niceness condition, we describe how any clarified $C_{p^n}$-Tambara field $k$ of characteristic $p$ may be constructed from suitably compatible clarified $C_{p^{n-1}}$-Tambara and $C_p$-Tambara fields of characteristic $p$, respectively $\ell_t$ and $\ell_b$ (along with one additional minor piece of information). Recursively, this reduces the classification to clarified $C_p$-Tambara fields $k$ of characteristic $p$. If the $C_p$ action on $k(C_p/e)$ is nontrivial, then it turns out that $k$ is again a fixed-point Tambara functor.

\begin{proposition}
	Let $k$ be a $C_p$-Tambara field and let $C_p$ act nontrivially on $k(C_p/e)$. Then the canonical map $k \rightarrow \underline{k(C_p/e)}$ is an isomorphism.
\end{proposition}

On the other hand, if the $C_p$ action on $k(C_p/e)$ is trivial, the classification is straightforward.

\begin{proposition}
	A clarified $C_p$-Tambara field of characteristic $p$ with trivial $C_p$-action on the bottom level $C_p/e$ is the same data as a choice of field $k(C_p/e)$ and subfield $k(C_p/C_p)$ which contains the image of the Frobenius endomorphism on $k(C_p/e)$.
\end{proposition}

In section 2, we review the necessary background on field-like Tambara functors and coinduction. Section 3 provides the reduction of the classification problem to clarified Tambara functors. Finally, section 4 explains how to construct any clarified $C_{p^n}$-Tambara functor from clarified $C_p$-Tambara functors, and classifies all clarified $C_p$-Tambara functors.

\subsection*{Acknowledgements}

The author would like to thank his advisor, Mike Hill, for many deep and insightful conversations. Additionally, the author thanks David Chan for sharing the proof of Proposition \ref{prop:Chan-Spitz}, due to David Chan and Ben Spitz. The author thanks Jason Schuchardt for noticing that the argument of Proposition \ref{prop:fields-of-char-not-p} applied even in when $k$ is not clarified, and David Chan for catching a mistake in an earlier draft. Next, the author thanks Alley Koenig for suggesting the name ``clarified". Finally, the author thanks Ben Szczesny and Haochen Cheng for helpful conversations.

\section{Recollections on Tambara functors}

For a complete introduction to Tambara functors, see \cite{Str12}. Recall that, for $G$ a finite group, a $G$-Tambara functor $k$ is roughly the following data:
\begin{enumerate}
	\item Rings $k(G/H)$ for each transitive $G$-set $G/H$. We say $k(G/H)$ is in \textit{level} $G/H$, and refer to $k(G/e)$ (resp. $k(G/G)$) as the \textit{bottom} level (resp. \textit{top}).
	\item Ring maps $k(G/H) \rightarrow k(G/K)$ for every morphism of $G$-sets $G/K \rightarrow G/H$.
	\item Multiplicative norm and additive transfer maps $k(G/H) \rightarrow k(G/K)$ for every morphism of $G$-sets $G/H \rightarrow G/K$.
\end{enumerate}

Note that the Weyl group $W_H = N_H/H$ of $H \subset G$ describes the automorphisms of the transitive $G$-set $G/H$, hence the rings $k(G/H)$ all possesses Weyl group actions, which are intertwined by the restriction maps. The norm, transfer, and restriction maps are required to satisfy various formulae. One of these is the double coset formula, which we describe here under the assumption that $G$ is abelian. For $H \subset L$, we have that the composition of the transfer $T_L^H : k(G/H) \rightarrow k(G/L)$ followed by restriction $R_L^H : k(G/L) \rightarrow k(G/H)$ is equal to the sum of the Weyl group orbits \[ R_L^H T_L^H = \Sigma_{g \in L/H} c_g \] where $c_g$ denotes the action of $g \in G \rightarrow W_H$ on $k(G/H)$. An analogous formula holds for the norm in place of the transfer, where the sum is replaced with a product.

Finally, given a Tambara functor $k$, we may identify it with the unique extension to a product-preserving functor from finite $G$-sets to rings; by product preserving, we mean $k(G/H \sqcup G/K) \cong k(G/H) \times k(G/K)$. This perspective will be required for the discussion of coinduction below.

\begin{example}
Let $R$ be a ring with $G$-action. The fixed points Tambara functor is the $G$-Tambara functor $\underline{R}$ with $\underline{R}(G/H) = R^H$. Noting that all restriction maps are inclusions, transfers and norms are uniquely defined as the appropriate sums (resp. products) of orbits via the double coset formula. The fixed point $G$-Tambara functor construction is functorial, and right adjoint to the functor $k \mapsto k(G/e)$ from $G$-Tambara functors to $G$-rings.
\end{example}

\begin{definition}[\cite{Nak11a}]
A nonzero $G$-Tambara functor $k$ is called field-like, or a $G$-Tambara field, if every nonzero morphism with domain $k$ is monic.
\end{definition}

By this definition, any field-like Tambara functor $k$ may be viewed as a subfunctor of the field-like Tambara functor $\underline{k(G/e)}$. This is because the adjunction unit $k \rightarrow \underline{k(G/e)}$ is nonzero, hence monic (hence injective in all levels). By this fact, to specify a field-like Tambara field, it is enough to specify a subring of each level of a Tambara field $\underline{R}$ which collectively are appropriately closed under taking transfers, norms, and restrictions.

\begin{proposition}[\cite{Nak11a}]
A $G$-Tambara functor $k$ is field-like if and only if all restriction maps are injective and $k(G/e)$ has no $G$-invariant ideals (recalling $W_e = G$).
\end{proposition}

Directly from this, we may prove the following result of David Chan and Ben Spitz. While this is an elementary consequence of the statement that $k(G/e)$ has no $G$-invariant ideals (in fact, it is equivalent to it), it greatly illuminates the structure of Tambara fields.

\begin{proposition}[\cite{CS24}]
\label{prop:Chan-Spitz}
	Let $k$ be a field-like $G$-Tambara functor. Then $k(G/e)$ is a product of copies of a field $\F$ permuted transitively by the $G$-action.
\end{proposition}

\begin{proof}
	Let $m$ be a maximal ideal of $k(G/e)$, and consider the $G$-set $\{ gm | g \in G \}$. Since $G$ acts transitively, it is isomorphic to $G/H$ for some subgroup $H \subset G$. Now consider $\cap_{gH \in G/H} gm$. This is a $G$-invariant ideal, hence by \cite[Theorem 4.32]{Nak11a} it must be the zero ideal. Writing $\F = k(G/e)/gm$ (which does not depend on the choice of $g$), by the Chinese remainder theorem, $k(G/e) \cong \F^{|G/H|}$. Since $G$ acts transitively on the $gm$, it transitively permutes the factors in the product.
\end{proof}

This result suggests the following definition, with which we reinterpret Nakaoka's result.

\begin{definition}
A $G$-ring is field-like if it has no nontrivial $G$-invariant ideals. Equivalently, it is a product of fields permuted transitively by the $G$ action.
\end{definition}

\begin{proposition}
A Tambara functor $k$ is field-like if and only if all restriction maps are injective and $k(G/e)$ is a field-like $G$-ring.
\end{proposition}

Without knowing Proposition \ref{prop:Chan-Spitz} or \cite[Theorem 4.32]{Nak11a}, it is a priori possible for a field-like Tambara functor $k$ with $k(G/e) \cong \Z/n$ to exist for some composite integer $n$. Fortunately, there is an intrinsic notion of characteristic of a $G$-Tambara functor, which by Proposition \ref{prop:Chan-Spitz} may be identified with the usual possibilities for characteristic of a field.

\begin{definition}
The characteristic of a Tambara functor $k$ is the equivalence class determined by the following equivalence relation: $k \sim \ell$ if $k \boxtimes \ell \neq 0$.
\end{definition}

\begin{corollary}
The characteristic of $k$ may be identified with the characteristic of $k(G/e)$.
\end{corollary}

\begin{proof}
Use the formula $(k \boxtimes \ell)(G/e) \cong k(G/e) \otimes \ell(G/e)$ and the fact that $k(G/e)$ and $\ell(G/e)$ are finite products of fields.
\end{proof}

There is likely interesting combinatorial structure captured by the box-product of field-like Tambara functors, although a more serious investigation falls outside the scope of this paper.

Finally, we review the coinduction functor. Given $H \subset G$, the coinduction $\textrm{Coind}_H^G \ell$ of an $H$-Tambara $\ell$ to a $G$-Tambara functor is obtained by precomposition with the restriction functor from finite $G$-sets to finite $H$-sets. This requires us to view $\ell$ as a functor on all finite $G$-sets, rather than merely the transitive ones. For $G = C_{p^n}$ and $\ell$ a $C_{p^k}$-Tambara functor, we supply a pictoral description of the coinduction $\textrm{Coind}_k^n \ell$ below:
\begin{align*}
\left( \textrm{Coind}_k^n \ell \right) \left( C_{p^n}/C_{p^n} \right) & \cong \ell(C_{p^k}/C_{p^k}) \\
\left( \textrm{Coind}_k^n \ell \right) \left( C_{p^n}/C_{p^{n-1}} \right) & \cong \ell(C_{p^k}/C_{p^k})^{\times p} \\
 & \vdots \\
\left( \textrm{Coind}_k^n \ell \right) \left( C_{p^n}/C_{p^{k+1}} \right) & \cong \ell(C_{p^k}/C_{p^k})^{\times p^{n-k-1}} \\
\left( \textrm{Coind}_k^n \ell \right) \left( C_{p^n}/C_{p^{k}} \right) & \cong \ell(C_{p^k}/C_{p^k})^{\times p^{n-k}} \\
\left( \textrm{Coind}_k^n \ell \right) \left( C_{p^n}/C_{p^{k-1}} \right) & \cong \ell(C_{p^k}/C_{p^{k-1}})^{\times p^{n-k}} \\
 & \vdots \\
\left( \textrm{Coind}_k^n \ell \right) \left( C_{p^n}/e \right) & \cong \ell(C_{p^k}/e)^{\times p^{n-k}} \textrm{.} \\
\end{align*}

One immediately observes using \cite[Theorem 4.32]{Nak11a} that if $\ell$ is a field-like $H$-Tambara functor, then so is $\textrm{Coind}_H^G \ell$. Coinduction is right adjoint to the restriction functor $\textrm{Res}_H^G$, which is given levelwise by precomposition with the coinduction functor from $H$-sets to $G$-sets \cite{Str12}. Heuristically, one may view coinduction as ``preserving the top level" and restriction as ``preserving the bottom level". In particular, restriction does not in general preserve Tambara fields, although we have the following.

\begin{proposition}
Suppose $k$ is a clarified $G$-Tambara field. Then for any subgroup $H \subset G$, $\textrm{Res}_H^G k$ is a clarified $H$-Tambara field.
\end{proposition}

There is also a coinduction functor from $H$-rings to $G$-rings, which is right adjoint to the restriction functor. It is given by $R \mapsto Fun(G/H,R)$, which we abbreviate by $\textrm{Coind}_H^G R$.

\section{Separated Tambara fields}

In this section we aim to reduce the classification of field-like $C_{p^n}$-Tambara functors to those whose bottom level $C_{p^n}/e$ is a field. To describe Tambara fields of this form, we introduce the notion of a clarified Tambara functor.

\begin{definition}
Let $k$ be a field-like Tambara functor. If $k(G/e) \cong \textnormal{Coind}_H^G R$ for some ring $R$ and proper subgroup $H \subset G$, we call $k$ separated. Otherwise, we call $k$ clarified. 
\end{definition}

\begin{lemma}
\label{lem:coind-commutes-with-underline}
	Let $R$ an $H$-ring. Then $\textnormal{Coind}_H^G \underline{R} \cong \underline{ \textnormal{Coind}_H^G R }$.
\end{lemma}

\begin{proof}
	Since coinduction is right adjoint to restriction and the fixed-point construction is right adjoint to the ``bottom-level" functor, it suffices to prove that the left adjoints commute, ie for all $G$-Tambara functors $k$, we have \[ \left( \textrm{Res}_H^G k \right) (H/e) \cong \textrm{Res}_H^G (k(G/e)) \textrm{.} \] Now the left-hand side is defined as $k \left( \textrm{Coind}_H^G (H/e) \right) \cong k(G/e)$ with $H$ acting through restriction of the $G$-action. This is precisely $\textrm{Res}_H^G (k(G/e))$, as desired.
\end{proof}

\begin{lemma}
Let $k$ a $G$-Tambara functor with $k(G/e) \cong \textnormal{Coind}_H^G R$ for some $H$-ring $R$, and suppose that the restriction map $k(G/H) \rightarrow k(G/e)$ is injective. Then we have an isomorphism $k(G/H) \cong \textnormal{Coind}_H^G S$ of rings for some subring $S \subset R$.
\end{lemma}

\begin{proof}
	Let $\{ x_{gH} \}$ denote the set of idempotents corresponding to projection on the each factor $R$ in level $G/e$. Note that this set is isomorphic to $G/H$. The double coset formula implies that the composition of the norm map from the bottom level $G/e$ to level $G/H$ with the restriction map to the bottom level sends each $x_{gH}$ to itself (the product over the $H$-orbits). By multiplicativity of the norm and injectivity of restriction, we see that the norms of the $x_{gH}$ form a complete set of orthogonal idempotents, which induce the desired isomorphism.
\end{proof}

\begin{corollary}
\label{cor:bottom-half-looks-coinduced}
Suppose $k$ is a field-like $G$-Tambara functor and $k(G/e) \cong \textnormal{Coind}_H^G \F$ for some $H$-field $\F$. Then whenever $L \subset H$, $k(G/L) \cong \textnormal{Coind}_H^G R$ for some subring $R$ of $\F$.
\end{corollary}

\begin{proof}
The restriction map $k(G/H) \rightarrow k(G/e)$ factors through $k(G/L)$, hence the sub-$G$-set of idempotents of $k(G/H)$ isomorphic to $G/H$ is also contained in $k(G/L)$
\end{proof}

\begin{lemma}
\label{lem:top-half-looks-coinduced}
	Suppose $G$ is abelian, $k$ is any $G$-Tambara functor such that $k(G/H)$ is a product of copies of some ring $R$ permuted freely and transitively by the Weyl group $G/H$, and $L$ is a subgroup of $G$ containing $H$ such that the restriction $k(G/L) \rightarrow k(G/H)$ is injective. Then the restriction map has image $k(G/H)^L$.
\end{lemma}

\begin{proof}
	Choose an idempotent $x \in k(G/H)$ corresponding to projection onto a factor $R$ and choose $r \in R$ arbitrary. The double coset formula implies that transferring $rx$ up to $k(G/L)$ and restricting the resulting element down to $k(G/H)$ results in \[ r \left( \Sigma_{g \in L/H} gx \right) \textrm{.} \] Repeating this process through all choices of $x$ and $r \in R$, we observe that the image of the restriction contains a collection of copies of $R$, embedded in $k(G/H)$ via the diagonal embedding $R \rightarrow R^{\times L/H}$ followed by any of the $|G/L|$ inclusions $R^{\times L/H} \hookrightarrow R^{\times G/H}$. Therefore the subring generated by the image is precisely the $L$-fixed points of $R^{\times G/H}$.
\end{proof}

The previous two lemmas show that any field-like $G$-Tambara functor ``looks like" a coinduced one in any level $G/L$ such that $L$ either contains or is contained in some fixed subgroup $H \subset G$. So, we can only deduce that field-like Tambara functors are always coinduced for families of abelian groups for which the subgroup lattice is a well-ordered set. This is why we only obtain a classification of fields for groups $C_{p^n}$. The author expects the following result to be true for abelian groups, and possibly even arbitrary finite groups, and intends to study this in forthcoming work.

\begin{theorem}
\label{thm:field-is-coinduced}
	If $k$ is a field-like $C_{p^n}$-Tambara functor, then $k \cong \textnormal{Coind}_s^n \ell$ for some clarified $C_{p^s}$-Tambara functor $\ell$.
\end{theorem}

\begin{proof}
	By Proposition \ref{prop:Chan-Spitz}, $k(C_p/e) \cong \textrm{Coind}_s^n \F$ for some $C_{p^s}$-field $\F$. Composing the canonical map to the fixed point Tambara functor of the $C_{p^n}/e$ level with the isomorphism of Lemma \ref{lem:coind-commutes-with-underline} supplies a map $k \rightarrow \textrm{Coind}_s^n \underline{\F}$ which is manifestly an isomorphism in level $C_{p^n}/e$.
	
	As rings, set $\ell(C_{p^s}/C_{p^j})$ to be the subring $R_j$ of $\F$ appearing in Corollary \ref{cor:bottom-half-looks-coinduced}, and identify $k(C_{p^n}/C_{p^j})$ with $\textrm{Coind}_s^n \ell(C_{p^s}/C_{p^j})$. The $C_{p^s}$-equivariant restriction maps for $\ell$ are obtained as the restriction of the restriction maps for $k$ to the $eC_{p^s}$ factor (the proof of Corollary \ref{cor:bottom-half-looks-coinduced} shows that this is well-defined). The norm and transfer maps are defined similarly, observing that the double coset formula along with injectivity of the restriction maps imply that the restriction of the norm (resp. transfer) in $k(C_{p^n}/C_{p^j})$ to the $eC_{p^s}$ factor lands in the $eC_{p^s}$ factor, for $j \leq s$. The exponential and double coset formulae for $k$ then become the double coset formulae for $\ell$. 
	
	We may alternatively construct $\ell$ as follows. Note that $\textrm{Res}_s^n k$ has an action of $C_{p^n}/C_{p^s}$ arising from the free and transitive permutation of the $C_{p^s}$-orbits of the $C_{p^s}$-sets \[ \textrm{Res}_s^n \left( \textrm{Coind}_s^n C_{p^s}/C_{p^k} \right) \] which corresponds in each level to permuting the $|C_{p^n}/C_{p^s}|$ factors $\ell(C_{p^s}/C_{p^j})$ of $k(C_{p^n}/C_{p^j})$. We define $\ell$ as the subfunctor of $\textrm{Res}_s^n k$ formed by the $C_{p^n}/C_{p^s}$-fixed points of this action.
	
	Now $\ell$ is a clarified field-like $C_{p^s}$-Tambara functor, and we may coinduce the canonical map \[ \ell \rightarrow \underline{\ell(C_{p^s}/e)} \] to \[ \textrm{Coind}_s^n \ell \rightarrow \textrm{Coind}_s^n \underline{\ell(C_{p^s}/e)} \cong \textrm{Coind}_s^n \underline{\F} \] Finally, the image of \[ k(C_{p^n}/C_{p^i}) \rightarrow \textrm{Coind}_s^n \underline{\F} \left(C_{p^n}/C_{p^i} \right) \] is precisely the image of $\textrm{Coind}_s^n \ell \left(C_{p^n}/C_{p^i} \right)$; when $i \leq k$ this is by construction of $\ell$, and when $i \geq k$ this is by Lemma \ref{lem:top-half-looks-coinduced}. Since $k$ and $\textrm{Coind}_s^n \ell$ are both field-like, they are naturally isomorphic to their images in $\textrm{Coind}_s^n \underline{\F}$, hence to each other.
\end{proof}

The author thanks Jason Schuchardt for pointing out that the following result (which in an earlier draft immediately preceeded Proposition \ref{prop:clarified-fields-with-perfect-fixed-points}) is true not just for clarified fields but for all fields, using the same argument.

\begin{proposition}
\label{prop:fields-of-char-not-p}
Suppose $k$ is a $G$-Tambara functor such that $|G|$ is invertible in the field $k(G/G)$. Then the canonical map $k \rightarrow \underline{k(G/e)}$ is an isomorphism.
\end{proposition}

\begin{proof}
Consider the restriction of the transfer map $k(G/e) \rightarrow k(G/H)$ to the $H$-fixed points. The double coset formula implies that postcomposing this map with the restriction $k(G/H) \rightarrow k(G/e)$ is multiplication by $|H|$, which is a unit in $k(G/e)$ by assumption. Therefore the restriction has image $k(G/e)^{H}$. Since it is injective by Nakaoka's theorem, it is an isomorphism $k(G/H) \cong k(G/e)^H$. This is precisely the statement that $k \rightarrow \underline{k(G/e)}$ is an isomorphism.
\end{proof}

\begin{corollary}
\label{cor:char-coprime-to-group-order-implies-coinduced}
Under the hypothesis of Proposition \ref{prop:fields-of-char-not-p}, $k \cong \textnormal{Coind}_H^G \ell$ for some clarified $H$-Tambara field $\ell$.
\end{corollary}

\begin{proof}
Combine Proposition \ref{prop:fields-of-char-not-p} with Lemma \ref{lem:coind-commutes-with-underline}.
\end{proof}

\begin{corollary}
The category of field-like $G$-Tambara functors of characteristic not dividing $|G|$ is adjointly equivalent to the category of field-like $G$-rings of characteristic not dividing the order of $G$.
\end{corollary}

\begin{proof}
The functor $R \mapsto \underline{R}$ is an inverse adjoint equivalence to the functor $k \mapsto k(G/e)$.
\end{proof}

\begin{corollary}
Let $k$ be a field-like $G$-Tambara functor of characteristic not dividing $|G|$. Then any morphism of field-like $G$-Tambara functors $\ell \rightarrow k$ which is an isomorphism on the bottom level $G/e$ is an isomorphism of Tambara functors.
\end{corollary}

This corollary may be of homotopical use. Namely, it heuristically says that the $G/e$ level homotopy group functor is conservative on $G$-$\mathbb{E}_\infty$-ring spectra whose homotopy groups are appropriately built out of field-like Tambara functors of characteristic not dividing $|G|$. We will see later that these corollaries fail in characteristic $p$.

\section{Clarified Tambara fields}

In this section we aim to classify the clarified Tambara fields. We observed in Proposition \ref{prop:fields-of-char-not-p} that the double coset formula forces many Tambara fields to be isomorphic to fixed-point Tambara functors. This idea continues to bear fruit.

Recall that Artin's lemma states that if a finite group $G$ acts on a field $\F$, then the inclusion of $G$-fixed points $\F^G \rightarrow \F$ is a Galois extension. The Galois group is the homomorphic image of $G$ in $\textrm{Aut}(\F)$.

\begin{proposition}
\label{prop:clarified-fields-with-perfect-fixed-points}
Suppose $k$ is a clarified $C_{p^n}$-Tambara functor such that the fixed-point field $k(C_{p^n}/e)^{C_{p^n}}$ is a perfect characteristic $p$ field. Then the canonical map $k \rightarrow \underline{k(C_{p^n}/e)}$ is an isomorphism.
\end{proposition}

\begin{proof}
As in the argument of Proposition \ref{prop:fields-of-char-not-p}, it suffices to show that each restriction map $k(C_{p^n}/C_{p^s}) \rightarrow k(C_{p^n}/e)$ has image $k(C_{p^n}/e)^{C_{p^s}}$. Since any Galois extension of a perfect field is perfect, our assumption ensures that each fixed-point field $k(C_{p^n}/e)^{C_{p^s}}$ is perfect.

Now consider the restriction of the norm $k(C_{p^n}/e) \rightarrow k(C_{p^n}/C_{p^s})$ to the $C_{p^s}$-fixed points. The double coset formula implies that postcomposing this map with the restriction $k(C_{p^n}/C_{p^s}) \rightarrow k(C_{p^n}/e)^{C_{p^s}}$ is $x \mapsto x^{p^s}$, ie the $s$-fold iterate of the Frobenius map. Since $k(C_{p^n}/e)^{C_{p^s}}$ is perfect, the Frobenius map is an isomorphism, so we observe that the restriction map is surjective, as desired.
\end{proof}

Combining Proposition \ref{prop:fields-of-char-not-p} with Proposition \ref{prop:clarified-fields-with-perfect-fixed-points}, we obtain Theorem \ref{thm:two-cases-of-clarified-fields}. Next, we analyze what happens when $k$ is a clarified Tambara field with $k(C_{p^n}/e)^{C_{p^n}}$ a possibly non-perfect characteristic $p$ field.

\begin{definition}
Let $k$ be a clarified $C_{p^n}$-Tambara functor of characteristic $p$. Writing $\phi$ for the Frobenius endomorphism, we call the subfield \[ \phi^n \left( k(C_{p^n}/e)^{C_{p^n}} \right) \] of $k(C_{p^n}/e)$ the lower bound field of $k$.
\end{definition}

\begin{proposition}
Suppose $k$ is a clarified $C_{p^n}$-Tambara functor of characteristic $p$. Then each $k(C_{p^n}/C_{p^s})$, viewed as a subring of $k(C_{p^n}/e)$, is an intermediate field of the extension $\phi^n \left( k(C_{p^n}/e)^{C_{p^n}} \right) \hookrightarrow k(C_{p^n})$
\end{proposition}

\begin{proof}
By the double coset formula, the lower bound field of $k$ is contained in the image of the composition of the norm map $k(C_{p^n}/e) \rightarrow k(C_{p^n}/C_{p^n})$ with the restriction $k(C_{p^n}/C_{p^n}) \rightarrow k(C_{p^n}/e)$. In particular, it is contained in the image of all restriction maps. Therefore each $k(C_{p^n}/C_{p^s})$ is a subring of $k(C_{p^n}/e)$ (via the restriction map) containing the lower bound field.

To show each $k(C_{p^n}/C_{p^s})$ is a field, it suffices to show each element has an inverse. Note that $k(C_{p^n}/e)$ is algebraic over the lower bound field, because it is a Galois extension of its $C_{p^n}$-fixed point subfield by Artin's lemma and any characteristic $p$ field is algebraic over the image of an iterate of the Frobenius endomorphism.

Letting $x \in k(C_{p^n}/C_{p^s})$, we see that $x$ is a root of some polynomial over the lower bound field. In particular, the subring of $k(C_{p^n}/C_{p^s})$ generated by $x$ and the lower bound field is a finite-dimensional vector space over the lower bound field, hence is Artinian. Since it is a subring of a field, it is an integral domain, hence a field. Thus $x$ has an inverse in $k(C_{p^n}/C_{p^s})$.
\end{proof}

Let $k$ be a clarified $C_{p^n}$-Tambara functor of characteristic $p$. Then we may construct a $C_{p^{n-1}}$-Tambara functor $\ell_t$ which captures the ``top piece" of $k$ as follows. Observe that for $s \geq 1$ each $k(C_{p^n}/C_{p^s})$ has a $C_{p^{n-1}}$ action with kernel $C_{p^{s-1}}$ (namely, regard the Weyl group as a quotient of $C_{p^n}$). First, set $\ell_t(C_{p^{n-1}}/C_{p^{s-1}}) = k(C_{p^n}/C_{p^s})$ for $1 \leq s \leq n$. Next, define the restriction maps for $\ell_t$ via the restriction maps for $k$. Since the restriction maps for $k$ are appropriately equivariant, so are those for $\ell_t$.

Finally, define the norm and transfer maps for $\ell_t$ via the norm and transfer maps for $k$ with the appropriate domain and codomain. To check that $\ell_t$ is a $C_{p^{n-1}}$-Tambara functor, it suffices to check that the appropriate double coset and exponential formulae are satisfied. In fact, we may do this in a universal example. Since we have already defined norms and transfers on $\ell_t$, via the map $k \rightarrow \underline{k(C_{p^n}/e)}$ it suffices to check that our construction produces a $C_{p^{n-1}}$-Tambara functor when applied to a fixed-point Tambara field $\underline{\F}$. This is clear, however, as our construction produces the fixed-point $C_{p^{n-1}}$-Tambara functor $\underline{\F^{C_{p^{n-1}}}}$.

On the other hand, we may extract a $C_p$-Tambara field $\ell_b$ from $k$ which recovers the ``bottom piece" of $k$ by $\ell_b := \textrm{Res}_1^n k$. Unwinding definitions, we have $\ell_b(C_p/e) = \textrm{Res}_1^n k(C_{p^n}/e)$ and $\ell_b(C_p/C_p) = \textrm{Res}_0^{n-1} k(C_{p^n}/C_p)$, with restriction, norm, and transfer for $k$ giving the restriction, norm, and transfer maps for $\ell_b$.

\begin{proposition}
Every clarified $C_{p^b}$-Tambara field $k$ of characteristic $p$ is obtained from the following: 
\begin{enumerate}
	\item a choice of clarified $C_{p^{n-1}}$-Tambara field $\ell_t$ of characteristic $p$
	\item a choice of $C_{p^n}$-field $\F = k(C_{p^n}/e)$
	\item a choice of clarified $C_p$-Tambara field $\ell_b$ of characteristic $p$.
\end{enumerate}
These choices must satisfy the following compatibility criteria:
\begin{enumerate}
	\item $\ell_b(C_p/C_p) = \textrm{Res}_0^{n-1} \ell_t(C_{p^{n-1}}/e)$
	\item $\ell_b(C_p/e) = \textrm{Res}_1^n \F$
	\item The ring map \[ \ell_t(C_{p^{n-1}}/e) \cong \ell_b(C_p/C_p) \rightarrow \ell_b(C_p/e) \cong \F \] is $C_{p^{n-1}}$-equivariant.
\end{enumerate}
\end{proposition}

\begin{proof}
Given $\ell_b$, $\ell_t$, and $k(C_{p^n}/e)$ as above, we define a $C_{p^n}$-Tambara functor $k$ as the following subfunctor of the fixed-point $C_{p^n}$-Tambara functor $\underline{\F}$. Set $k(C_{p^n}/e) = \F$ and $k(C_{p^n}/C_{p^s}) = \ell_t(C_{p^{n-1}}/C_{p^{s-1}})$ for $s \geq 1$. The restrictions, norms, and transfers which do not factor nontrivially through $C_{p^n}/C_p$ are well-defined (in the sense that their codomain contains their image) because they are well-defined for $\ell_t$ and $\ell_b$ respectively. The remaining restrictions, norms, and transfers are well-defined because they are compositions of well-defined restrictions, norms, and transfers, respectively.
\end{proof}

This recursively reduces the classification of clarified $C_{p^n}$-Tambara fields of characteristic $p$ to clarified $C_p$-Tambara fields of  characteristic $p$. Let $k$ be such a $C_p$-Tambara functor. If $C_p$ acts trivially on $k(C_p/e)$, then the composition of the norm map with the restriction may be identified with the Frobenius endomorphism, and the transfer map is zero. Thus $k(C_p/C_p)$ may be any subfield of $k(C_p/e)^{C_p}$ containing the image of the Frobenius endomorphism. Therefore we obtain the following.

\begin{proposition}
	A clarified $C_p$-Tambara field of characteristic $p$ with trivial $C_p$-action on the bottom level $C_p/e$ is the same data as a choice of field $k(C_p/e)$ and subfield $k(C_p/C_p)$ which contains the image of the Frobenius endomorphism on $k(C_p/e)$.
\end{proposition}

\begin{example}
We may form a $C_p$-Tambara functor of the above type as follows. First, consider the fixed-point Tambara functor associated to the trivial $C_p$ action on $\F_p(t)$. We may form a sub-Tambara functor with the same bottom level $C_p/e$, but top level equal to the image of the Frobenius endomorphism $\F_p(t^p)$. The inclusion of this sub-functor provides an example of a morphism between Tambara fields which is an isomorphism on the bottom level, but is not an isomorphism.
\end{example}

On the other hand, when the $C_p$-action is nontrivial, it turns out that $k$ must again be a fixed-point functor. Note that if $C_p$ acts nontrivially on $k(C_p/e)$, then it acts faithfully.

\begin{proposition}
	Let $k$ be a field-like $G$-Tambara functor and let $H$ be the kernel of the group homomorphism $G \rightarrow \mathrm{Aut(k(G/e))}$ specified by the $G$-action on $k(G/e)$ (so $G/H$ acts faithfully on $k(G/e)$). Then the restriction $k(G/L) \rightarrow k(G/e)^{L/H}$ is an isomorphism for any normal subgroup $L$ of $G$ containing $H$.
\end{proposition}

\begin{proof}
	By assumption $k(G/e)$ is a Galois extension of $k(G/e)^{L/H}$ with Galois group $G/L$. As usual it suffices to show that the restriction is surjective. Note that our assumptions imply that the double coset formula describes the Galois-theoretic trace. In other words, it suffices to observe that the trace map $k(G/e) \rightarrow k(G/e)^{L/H}$ is surjective. But this follows from \cite[Chapter VI, Theorem 5.2]{Lan02} since Galois extensions with finite Galois group are finite and separable.
\end{proof}

Heuristically, this proposition says that any $C_{p^n}$-Tambara field of characteristic $p$ looks like a fixed-point Tambara functor in all levels below a certain point, depending on the kernel of the $C_{p^n}$ action on the bottom level. Above that point, the Tambara field can have ``fixed point jumps" where the restriction between adjacent levels fails to surject onto the fixed points. For $C_p$ we obtain the following.

\begin{corollary}
	Let $k$ be a $C_p$-Tambara field and let $C_p$ act nontrivially on $k(C_p/e)$. Then the canonical map $k \rightarrow \underline{k(C_p/e)}$ is an isomorphism.
\end{corollary}

This concludes the classification of field-like $C_{p^n}$-Tambara functors.

\bibliographystyle{alpha}
\phantomsection\addcontentsline{toc}{section}{\refname}
\bibliography{ref}
\end{document}